\documentclass{article}
\usepackage[letterpaper]{geometry}

\usepackage{amsmath}
\usepackage{amssymb}
\usepackage{amsthm}
\usepackage{enumerate}
\usepackage{color}
\usepackage{url}

\allowdisplaybreaks[1]

\newtheorem{theorem}{Theorem}[section]
\newtheorem{corollary}[theorem]{Corollary}
\newtheorem{lemma}[theorem]{Lemma}
\newtheorem{example}[theorem]{Example}
\newtheorem{proposition}[theorem]{Proposition}
\newtheorem{definition}[theorem]{Definition}

\newcommand{\Z}{\mathbb{Z}}

\newcommand{\GF}{{\rm{GF}}}
\newcommand{\bone}{{\mathbf{1}}}

\long\def\symbolfootnote[#1]#2{\begingroup%
\def\thefootnote{\fnsymbol{footnote}}\footnote[#1]{#2}\endgroup}

\begin{document}

\title{Linking systems of difference sets}
\author{Jonathan Jedwab \and Shuxing Li \and Samuel Simon}
\date{14 August 2017 (revised 20 April 2018)}
\maketitle

\symbolfootnote[0]{
Department of Mathematics, 
Simon Fraser University, 8888 University Drive, Burnaby BC V5A 1S6, Canada.
\par
J.~Jedwab is supported by NSERC.
\par
Email: {\tt jed@sfu.ca}, {\tt shuxing\_li@sfu.ca}, {\tt ssimon@sfu.ca}
\par
The results of this paper form part of the Master's thesis of S. Simon \cite{simon-masters}, who presented them in part at the CanaDAM 2017 conference in Toronto, ON.
}

\begin{abstract}
A linking system of difference sets is a collection of mutually related group difference sets, whose advantageous properties have been used to extend classical constructions of systems of linked symmetric designs. 
The central problems are to determine which groups contain a linking system of difference sets, and how large such a system can be.
All previous constructive results for linking systems of difference sets are restricted to 2-groups.
We use an elementary projection argument to show that neither the McFarland/Dillon nor the Spence construction of difference sets can give rise to a linking system of difference sets in non-$2$-groups.
We make a connection to Kerdock and bent sets, which provides large linking systems of difference sets in elementary abelian $2$-groups.
We give a new construction for linking systems of difference sets in 2-groups, taking advantage of a previously unrecognized connection with group difference matrices. This construction simplifies and extends prior results, producing larger linking systems than before in certain 2-groups, new linking systems in other 2-groups for which no system was previously known, and the first known examples in nonabelian groups.
\end{abstract}

\section{Introduction}
\label{sec:introduction}

\subsection{Difference sets}
\label{subsec:diffsets}

The study of difference sets lies at the intersection of combinatorics, finite geometry, and coding theory \cite{jungnickel-survey}, \cite{jungnickel-survey-update}, \cite{jungnickel-survey-update2}.
The advantageous structural properties of difference sets enable the solution of problems in radar, optical image alignment, and other areas of digital communication~\cite{unify}. 
Difference sets occur within the larger context of the theory of experimental design: a difference set in a group $G$ is equivalent to a symmetric design with a regular automorphism group~$G$~\cite{lander}. 

\begin{definition}
Let $G$ be a group of order $v$, written multiplicatively, and let $D$ be a subset of $G$ with $k$ elements. Then $D$ is a $(v,k,\lambda,n)$-\emph{difference set in $G$} if the multiset $\{d_1 d_2^{-1}: d_1, d_2 \in D \text{ and } d_1 \ne d_2 \}$ contains every non-identity element of $G$ exactly $\lambda$ times. In this case, we define $n:=k- \lambda$.  
\end{definition}

The central problems are to determine which groups contain a difference set, and to enumerate all inequivalent examples in such groups. The cases $k=0$ and $k=1$ are considered trivial, and by taking the complement of a difference set if necessary we may assume that $k \le v/2$. 
We shall be concerned with the following three parameter families
$(v, k, \lambda, n)$ of difference sets,  where $q$ is a prime power, $d$ is a nonnegative integer, and $N$ is a positive integer.

\begin{center}
\begin{tabular}{c | c | c | c | c  }
Family& $v$&$k$&$\lambda$&$n$\\
\hline\\[-4mm]
McFarland & $q^{d+1} \left( \frac{q^{d+1}-1}{q-1}+1\right) $& $ q^d \left(\frac{q^{d+1}-1}{q-1} \right)$ &
 $q^d \left(\frac{q^{d}-1}{q-1} \right)$ &$ q^{2d} $\\[2mm]
Spence & $3^{d+1} \left(\frac{3^{d+1}-1}{2} \right)$&$ 3^d \left(\frac{3^{d+1}+1}{2} \right)$&$ 3^d \left(\frac{3^{d}+1}{2} \right)$&$ 3^{2d} $ \\[2mm]
Hadamard &  $4N^2$&$ N(2N-1)$&$N(N-1)$&$N^2$  
\end{tabular}
\end{center}

The McFarland parameters with $q=2$ are the same as the Hadamard parameters with $N=2^d$, and the corresponding difference sets occur in $2$-groups. 
Theorem~\ref{2grp} shows that the parameters of all (nontrivial) difference sets in $2$-groups must take this common form.

\begin{theorem}[{\cite[Chapter II, Theorem~3.17]{bjl2}}]\label{2grp}
Suppose a group $G$ of order $2^{r}$ contains a $(v,k,\lambda, n)$-difference set where $2 \le k \le \frac{v}{2}$. Then $r=2d+2$ for some $d \ge 0$ and 
\begin{equation*}
(v,k, \lambda,n) = \left( 2^{2d+2}, 2^d(2^{d+1}-1), 2^d(2^d-1),2^{2d} \right).
\end{equation*}
\end{theorem}
Theorem~\ref{NecSuf} gives necessary and sufficient conditions for the existence of a difference set in an abelian $2$-group.
(The \emph{exponent} of a group $G$ with identity $\bone_G$ is the smallest positive integer $\alpha$ for which $g^\alpha=\bone_G$ for all $g \in G$, and is written $\exp(G)$.)
\begin{theorem}[Kraemer \cite{kraemer}] \label{NecSuf}
A difference set exists in an abelian group $G$ of order $2^{2d+2}$ if and only if $\exp(G) \le 2^{d+2}$.
\end{theorem}

By a common notation convention, we identify a multiset of elements of the group $G$ with its corresponding element in the group ring $\Z[G]$.
Given a multiset $S$ of elements of the group $G$, we write $S^{(-1)}$ for the group ring element $\sum_{s\in S} s^{-1}$, where the sum is over the elements in the multiset $S$ and the inverse is taken in $G$. 
The following result is then a direct consequence of the definition of a difference set and the relation $n=k- \lambda$. 
\begin{lemma}\label{DsetProp}
Let $G$ be a group of order $v$ and $D$ a subset of $G$ with $k$ elements. Then $D$ is a $(v,k,\lambda,n)$-difference set in $G$ if and only if
\begin{align}\label{Dset}
DD^{(-1)} =n \bone_G+ \lambda G \quad \text{ \rm{in} }\Z[G].
\end{align}
\end{lemma}

\subsection{Linking systems of difference sets}

Linking systems of difference sets were introduced by Davis, Martin, and Polhill~\cite{davis-martin-polhill}. Such a system gives rise to a system of linked symmetric designs, as introduced by Cameron \cite{cameron-doubly} and studied by Cameron and Seidel~\cite{cameron-seidel}, and is equivalent to a 3-class Q-antipodal cometric association scheme~\cite{vandam}. 
Kodalen \cite{kodalen} recently constructed the first known examples of systems of linked symmetric designs in non-$2$-groups, but it remains an important open question as to whether linking systems of difference sets can exist in non-$2$-groups.
 
\begin{definition}\label{link}
Let $G$ be a group of order $v$, written multiplicatively, and let $ \ell \ge 2$. Suppose $\mathcal{L}= \{D_{i,j} : 0 \le i,j  \le  \ell \text{ and }i \ne j \}$ is a collection of size $\ell (\ell+1)$ of $(v,k,\lambda,n)$-difference sets in $G$. Then $\mathcal{L}$ is a $(v,k,\lambda,n; \ell+1)$\emph{-linking system of difference sets in $G$} if there are integers $\mu, \nu$ such that for all distinct $i,j,h$, the following equations hold in $\Z[G]$:
\begin{align}
D_{h,i} D_{i,j} &= (\mu-\nu) D_{h,j} + \nu G \label{Cond1} \\
 D_{i,j} &= D_{j,i}^{(-1)} .\label{Cond2}
\end{align}
\end{definition}
The central problems are to determine which groups contain a linking system of difference sets, and how large such a system can be. 
Definition \ref{link} is rather cumbersome. We define a simpler object in Definition \ref{reduced} and show in Proposition \ref{equiv} (whose proof we postpone until Section~\ref{sec:equiv}) that it is equivalent to a linking system of difference sets. 
\begin{definition}\label{reduced}
Let $G$ be a group of order $v$, written multiplicatively, and let $\ell \ge 2$. Suppose $\mathcal{R}=\{D_1, D_2, \cdots, D_\ell \}$ is a collection of size $\ell$ of $(v,k, \lambda,n)$-difference sets in $G$. Then $\mathcal{R}$ is a \emph{reduced $(v,k,\lambda,n; \ell)$-linking system of difference sets in $G$ of size $\ell$} if there are integers $\mu,\nu$ such that for all distinct $i,j$ there is some $(v,k,\lambda,n)$-difference set $D(i,j)$ in $G$ satisfying
\begin{align}\label{LinkingProperty}
D_iD_j^{(-1)}= (\mu- \nu) D(i,j) + \nu G \quad \text{ in } \Z[G].
\end{align}
\end{definition}
Note that the difference set $D(i,j)$ in \eqref{LinkingProperty} is not necessarily contained in the collection~$\mathcal{R}$. 
Note also that $\{D_1, D_2, \dots, D_\ell\}$ is a reduced $(v,k,\lambda,n;\ell)$-linking system of difference sets in $G$ with respect to $\mu,\nu$ if and only if $\{G-D_1, G-D_2, \dots, G-D_\ell\}$ is a reduced $(v,v-k,v-2k+\lambda,n;\ell)$-linking system of difference sets in $G$ with respect to $v-2k+\nu,\,v-2k+\mu$, so we may assume that $k \le v/2$.

\begin{proposition}\label{equiv}
Let $\mu, \nu$ be integers. 
A $(v,k,\lambda, n;\ell+1)$-linking system of difference sets in a group $G$ with respect to $\mu, \nu$ is equivalent to a reduced $(v,k, \lambda,n; \ell)$-linking system of difference sets in $G$ with respect to $\mu, \nu$. 
\end{proposition}

A difference set $D$ satisfying $D=D^{(-1)}$ is called \emph{reversible}.  

\begin{example}[{\cite[Example~6.3]{davis-martin-polhill}}]\label{FromDMP}
Let $G= \Z_4^2= \langle x,y \rangle$ and let
$D_1 = x + x^3y + y^3 + x^3 + xy^3 + y $, 
$D_2 = x + x^3y + y^3 + xy^2 + xy + x^2y $, 
$D_3 = x + x^3y + y^3 + x^2y^3 + x^3y^3 + x^3y^2$.
For each $i$ we have $D_iD_i^{(-1)}= 4 \cdot \bone_G+2G$, so $D_i$ is a difference set in~$G$ by Lemma~\ref{DsetProp}. Furthermore
\[
D_2 D_1^{(-1)} = -2D+ 3G,
\]
where $D= y^3+x+x^2y^3+x^3y+x^3y^2+x^3y^3$ is a $(16,6,2,4)$-difference set in $G$. Similar calculation for $D_i D_j^{(-1)}$ for each distinct $i, j$ shows that $\{D_1,D_2,D_3 \}$ forms a reduced $(16,6,2,4;3)$-linking system of difference sets in $G$. 
The difference set $D_1$ is reversible, but neither $D_2$ nor $D_3$ is. 
\end{example}

\begin{definition}
Suppose $\mathcal{L}= \{D_{i,j} : 0 \le i,j  \le \ell \text{ and }i \ne j \}$ is a $(v,k,\lambda,n; \ell+1)$-linking system of difference sets in a group $G$. If each difference set $D_{i,j}$ is reversible, then $\mathcal{L}$ is a \emph{reversible $(v,k,\lambda,n; \ell+1)$-linking system of difference sets in $G$}. 
\end{definition}

\begin{definition}
Suppose $\mathcal{R}=\{D_1, D_2, \cdots, D_\ell \}$ is a reduced $(v,k,\lambda,n; \ell)$-linking system of difference sets in a group $G$. If the corresponding linking system $\mathcal{L}$ (as defined in the proof of Proposition \ref{equiv}) is reversible, then $\mathcal{R}$ is a \emph{reversible reduced $(v,k,\lambda,n; \ell)$-linking system of difference sets in $G$}.
\end{definition}

\subsection{Results due to Davis-Martin-Polhill}\label{DMPresults}

Davis, Martin, and Polhill \cite{davis-martin-polhill} provide one of the two principal references on linking systems of difference sets. All their results construct reduced linking systems of difference sets in abelian $2$-groups, and all their examples are reversible except for the one presented here as Example~\ref{FromDMP}.

The main result of \cite{davis-martin-polhill} depends on several theorems: 
a base construction \cite[Theorem~5.3]{davis-martin-polhill} for a reversible reduced linking system of difference sets from partial difference sets having intricate mutual properties; 
a product construction \cite[Theorem~3.1]{davis-martin-polhill} for combining two reversible reduced linking systems of Hadamard difference sets into a larger one;
and a construction for a reversible reduced linking system of difference sets 
using Galois rings \cite[Theorem~4.6]{davis-martin-polhill}. These are combined to give the following result (in which we have corrected some typographical errors and oversights in \cite{davis-martin-polhill} following private communication with the authors).
Note from Theorem \ref{2grp} that the parameters of a difference set in a $2$-group are determined by the order of the group. 

\begin{theorem}[Davis, Martin, and Polhill {\cite[Corollary~5.5]{davis-martin-polhill}}]\label{thm-dmp}
Let $G=\Z_{2^{a_1}}^{2 b_1} \times \cdots \times \Z_{2^{a_k}}^{2 b_k}$ for integers $a_i,b_i$ satisfying $a_i \ge 1$ and $b_i \ge 2$, and let $b \ge 2$. Then the groups below contain a reversible reduced linking system of the specified size.
\begin{center}
\begin{tabular}{c | c }
\rm{Group} & \rm{Size}  \\
\hline\\[-3mm]
$G$ & $2^{\min(b_1,b_2, \dots ,b_k)}-1$  \\[1mm]
$\Z_4^b$ & $2^{b}-1$ \\[2mm]
$G \times \Z_4^b$ & $2^{\min(b_1,b_2, \dots ,b_k,b)}-1$ 
\end{tabular}
\end{center}
\end{theorem}

\subsection{Overview of paper}
\label{overview}

The paper \cite{davis-martin-polhill}) concludes with five open problems, of which we shall address the following (originally numbered 1, 3, 4, and 5).

\begin{itemize}
\item[Q1.]  \emph{Investigate the relationships between the difference set constructions of linked systems [given in \cite{davis-martin-polhill}] with the constructions of the Cameron-Seidel family and the Kerdock codes.}
\item[Q2.] \emph{Can difference sets be used to construct systems of linked designs with different parameters, for instance in the Hadamard family $(4N^2, 2N^2-N,N^2-N)$ but with $N$ not a power of 2?}
\item[Q3.] \emph{Is there an infinite family that generalizes [Example \ref{FromDMP}]?}
\item[Q4.] \emph{Can [generalizations of difference sets] be exploited to find other linked systems of mathematical structures?}
\end{itemize}

In Section~\ref{sec:equiv} we prove the equivalence stated in Proposition~\ref{equiv}, which requires particular care when the group $G$ is nonabelian.

In Section \ref{sec:Boolean} we reinterpret previous work on systems of linked symmetric designs and bent sets in order to produce a large reduced linking system of difference sets in the elementary abelian group~$\Z_2^{2d+2}$, giving a partial answer to Q1.

In Section~\ref{sec:nonexistence} we uncover an obstruction to the existence of a reduced linking system of McFarland difference sets having $q>2$, and a reduced linking system of Spence difference sets, using only elementary arguments that depend on a well-chosen modular reduction in the group ring. Since the associated groups are non-$2$-groups, this provides a partial answer to~Q2.

In Section~\ref{sec:construction} we seek further constructions in $2$-groups. Our main construction (Theorem~\ref{Main}) relies on the unexpected use of group difference matrices, which addresses~Q4.
We derive multiple corollaries of this construction, as summarized in Table \ref{Summary} and illustrated in Table~\ref{Compare} for abelian groups of order 64. Tables~\ref{Summary} and~\ref{Compare} also include the constructive result of Section~\ref{sec:Boolean}. (By Theorem~\ref{NecSuf}, we need not consider groups of exponent greater than $2^{d+2}$ in Table~\ref{Summary}, nor those of exponent greater than $16$ in Table~\ref{Compare}. An abelian $2$-group is isomorphic to 
$ \Z_{2^{a_1}} \times  \Z_{2^{a_2}} \times \cdots \times  \Z_{2^{a_t}}$ for some integers $a_i$ and $t$, and its \emph{rank} is then~$t$.)
We construct an infinite family of examples in nonabelian groups, whereas not a single nonabelian example was previously known. We obtain an infinite family of nonreversible examples generalizing Example \ref{FromDMP}, answering~Q3.
Finally, we give a detailed examination of reduced linking systems of difference sets in~$\Z_4^2$.

In Section~\ref{sec:open} we suggest directions for further research by posing several open problems. 

\begin{table}[h]
\begin{center}\caption {Constructions of a reduced linking system of difference sets in an abelian group $G$ of order $2^{2d+2}$, rank at least $d+1$, and exponent~$2^e$. }\label{Summary}
\vspace{2mm}
\begin{tabular}{c | c | c}
Range of $e$ & Size of system &  Source \\
\hline\\[-3mm]
1 & $2^{2d+1}-1$ & Bent set ( Corollary \ref{bent=link} ) \\[1mm] 
$[2 , \frac{d+3}{2}]$ & $2^{ \left\lfloor \frac{d+1}{e-1} \right\rfloor }-1$ & Difference matrix (Corollary \ref{improved} ) \\[2mm]
$(\frac{d+3}{2}, d+1]$ & 3 & Difference matrix (Corollary \ref{general} )\\[1mm]
$d+2$  (so $G= \Z_{2^{d+2}} \times \Z_2^d)$ & No result & 
\end{tabular}
\end{center}
\end{table}

\begin{table}[h]
\begin{center}
\caption{Comparison of maximum known sizes of reduced linking systems of difference sets in abelian groups of order $64$.}\label{Compare}
\vspace{2mm}
\begin{tabular}{c | c | c | c}
	& Previous 	& Current 	&        \\
Group   & maximum 	& maximum	& Source \\
        & known size 	& known size   	&        \\
\hline\\[-3mm]
$\Z_2^6$ & 31 \cite{bey-kyureghyan} & 31 & Bent set (Corollary \ref{bent=link}) \\
$\Z_4 \times \Z_2^4$ & None &  {7} & Difference matrix (Corollary \ref{improved}) \\
$\Z_4^2 \times \Z_2^2$ & None &  {7} & Difference matrix (Corollary \ref{improved}) \\
$\Z_4^3$ & 7 \cite{davis-martin-polhill} & 7 & Difference matrix (Corollary \ref{improved}) \\
$\Z_8 \times \Z_2^3$ & None & {3} & Difference matrix  (Corollary \ref{general})\\
$\Z_8 \times \Z_4 \times \Z_2$ & None & {3} & Difference matrix (Corollary \ref{general}) \\
$\Z_8^2$ & None & None &\\
$\Z_{16} \times \Z_2^2$ & None & None & \\
$\Z_{16} \times \Z_4$ & None & None &
\end{tabular}
\end{center}
\end{table}

\section{Proof of Proposition~\ref{equiv}}
\label{sec:equiv}
In this section we prove Proposition~\ref{equiv}. An outline of the main argument of the proof is implicit in \cite{davis-martin-polhill}, although many details are omitted there and particular care is needed when the group $G$ is nonabelian. 
We firstly use the classical result of Proposition~\ref{straightforward} to show in Lemma~\ref{MuNuReduced} that the parameters $\mu$, $\nu$ in a reduced linking system of difference sets are determined to within a sign. Lemma~\ref{MuNuReduced} corresponds to a result stated by Noda \cite[Proposition~0]{noda} for systems of linked symmetric designs.

\begin{proposition}[{\cite[p.~468]{bruck}}]\label{straightforward}
Suppose $D$ is a $(v,k,\lambda,n)$-difference set in a (not necessarily abelian) group~$G$. Then $D^{(-1)}$ is also a $(v,k,\lambda,n)$-difference set in $G$.
\end{proposition}

\begin{lemma}\label{MuNuReduced}
Suppose $\{D_1, D_2, \dots, D_\ell\}$ is a reduced $(v,k,\lambda,n;\ell)$-linking system of difference sets in a group~$G$ with respect to integers $\mu$, $\nu$. Then
\[
\nu=\frac{ k(k\pm \sqrt{n})}{v} \text{ and } \mu = \nu \mp \sqrt{n}.
\]
\end{lemma}
\begin{proof}
Choose distinct $i$, $j$ satisfying $1 \le i, j \le \ell$.
By Definition~\ref{reduced}, there is a $(v, k, \lambda, n)$-difference set $D(i,j)$ in $G$ such that
\begin{equation}\label{DiDj-1}
D_i D_j^{(-1)} = (\mu-\nu)D(i,j) + \nu G \quad \mbox{in $\Z[G]$},
\end{equation}
and so
\begin{align} 
\Big((\mu-\nu)D(i,j)\Big) \Big((\mu-\nu)D(i,j)^{(-1)}\Big)
 &= \big(D_i D_j^{(-1)}-\nu G\big) \big(D_j D_i^{(-1)}-\nu G\big) \nonumber \\
 &= D_i \big(D_j^{(-1)}D_j\big) D_i^{(-1)} -2\nu k^2 G +\nu^2 v G \label{oversight}
\end{align}
because, for a subset $S$ of $G$, we have $SG = GS = |S|G$ in $\Z[G]$.
Now $D_i$ and $D_j$ and $D(i,j)$ are each $(v,k,\lambda,n)$-difference sets in $G$, and by Proposition~\ref{straightforward} so is~$D_j^{(-1)}$. Using Lemma~\ref{DsetProp} we therefore find from \eqref{oversight} that
\begin{align*}
(\mu-\nu)^2 (n \bone_G+\lambda G)
 &= D_i(n\bone_G+\lambda G)D_i^{(-1)} - 2\nu k^2G+\nu^2 v G \\
 &= (n\bone_G+\lambda G)^2 - 2\nu k^2G+\nu^2 v G.
\end{align*}
Since the coefficients of $G-\bone_G$ on both sides must be equal, comparison of the coefficients of $\bone_G$ shows that
\begin{equation}\label{sqrtn}
\mu-\nu = \mp \sqrt{n}.
\end{equation}
Counting terms on both sides of \eqref{DiDj-1} then gives
\begin{align*}
k^2 &= (\mu-\nu) k + \nu v \\
    &= \mp \sqrt{n}\, k + \nu v,
\end{align*}
which together with \eqref{sqrtn} establishes the required values for $\mu$ and~$\nu$.
\end{proof}

We can now prove Proposition~\ref{equiv}.

\begin{proof}[Proof of Proposition~\ref{equiv}]
Let $\mathcal{L}= \{D_{i,j} : 0 \le i,j  \le \ell \text{ and }i \ne j \}$ be a $(v,k,\lambda,n; \ell+1)$-linking system of difference sets in $G$ with respect to $\mu, \nu$. Let $D_i = D_{i,0}$ for $1 \le i \le \ell$ and let $\mathcal{R} = \{D_1, D_2, \cdots, D_\ell\}$. Then for all distinct $i,j$,
\begin{align*}
D_{i} D_{j}^{(-1)} = D_{i,0}D_{j,0}^{(-1)}= D_{i,0}D_{0,j}= (\mu- \nu) D_{i,j} + \nu G
\end{align*}
using \eqref{Cond2} and \eqref{Cond1}. Therefore $\mathcal{R}$ is a reduced $(v,k,\lambda,n; \ell)$-linking system of difference sets in $G$ with respect to $\mu, \nu$. 

Conversely, let $\mathcal{R}= \{D_1, D_2, \cdots, D_ \ell \}$ be a reduced $(v,k,\lambda,n;\ell)$-linking system of difference sets in $G$ with respect to $\mu, \nu$. Let $D_{i,0}=D_i$ and $D_{0,i} = D_i^{(-1)}$ for $1 \le i \le \ell$. For distinct $i,j$ not equal to 0, let $D_{i,j}$ be the difference set $D(i,j)$ given by Definition \ref{reduced} applied to $D_{i,0}$ and $D_{j,0}$, so that
\begin{equation}\label{linkedprop}
D_{i,0}D_{j,0}^{(-1)} = (\mu-\nu) D_{i,j} +\nu G.
\end{equation}
We shall show that $\mathcal{L}= \{D_{i,j} : 0 \le i,j  \le \ell \text{ and }i \ne j \}$ is a $(v,k,\lambda, n ; \ell+1)$-linking system of difference sets in $G$ with respect to $\mu, \nu$ by showing that \eqref{Cond1} and \eqref{Cond2} hold. 

To show \eqref{Cond2} for distinct $i,j$, one of which is 0, use the definition of $D_{i,0}$ and $D_{0,i}$. 
To show \eqref{Cond2} for distinct $i,j$, both of which are not 0, apply the operation $^{(-1)}$ to both sides of \eqref{linkedprop} to obtain
\begin{equation*}
D_{j,0} D_{i,0}^{(-1)} = (\mu-\nu)D_{i,j}^{(-1)} + \nu G.
\end{equation*}
Interchange $i,j$ to get
\begin{equation*}
D_{i,0} D_{j,0}^{(-1)} = (\mu-\nu)D_{j,i}^{(-1)} + \nu G.
\end{equation*}
By comparison with \eqref{linkedprop}, we conclude that $D_{i,j}= D_{j,i}^{(-1)}$, giving \eqref{Cond2}.

To show \eqref{Cond1} for distinct $i,j,h$, all of which are not 0, use \eqref{linkedprop} to form the product
\begin{align}
\Big(  (\mu-\nu)D_{h,i}  \Big) \Big(  (\mu-\nu)D_{i,j}  \Big) 
 &= \left( D_{h,0}D_{i,0}^{(-1)} -\nu G \right) \left( D_{i,0}D_{j,0}^{(-1)} -\nu G \right) \nonumber \\
 &= D_{h,0}(D_{i,0}^{(-1)}D_{i,0})D_{j,0}^{(-1)}- 2 \nu k^2 G+ \nu^2 v G. \label{systemprod}
\end{align}
 From Lemma~\ref{MuNuReduced} we have $(\mu-\nu)^2=n$. Since $D_{i,0}$ is a $(v,k,\lambda,n)$-difference set in $G$, by Proposition \ref{straightforward} so is~$D_{i,0}^{(-1)}$. Using Lemma \ref{DsetProp} we therefore find from \eqref{systemprod} that
\begin{align*}
nD_{h,i}D_{i,j}
 &=D_{h,0}(n\bone_G+ \lambda G)D_{j,0}^{(-1)}- 2 \nu k^2 G+ \nu^2 v G \\
 &=nD_{h,0}D_{j,0}^{(-1)}+ (\lambda k^2 - 2 \nu k^2 + \nu^2 v) G.
\end{align*}
Counting terms on both sides shows that $\lambda k^2 - 2 \nu k^2 + \nu^2 v = 0$, so that
\begin{align*}
D_{h,i}D_{i,j}&= D_{h,0}D_{j,0}^{(-1)} \\
&=( \mu - \nu) D_{h,j}+ \nu G
\end{align*}
using \eqref{linkedprop} again, as required for \eqref{Cond1}.

It remains to show \eqref{Cond1} for distinct $i,j,h$, exactly one of which is 0. The case $i=0$ follows from the definition of $D_{h,j}$. We now outline the case $h=0$; the case $j=0$ is similar. From 
\eqref{linkedprop} we have 
\begin{align}
(\mu - \nu) D_{0,i} D_{i,j} &= D_i^{(-1)}(D_i D_j^{(-1)}- \nu G) \nonumber \\
&=(n \bone_G + \lambda G) D_j^{(-1)} - \nu k G \nonumber \\
&= n D_{0,j} + (\lambda- \nu)kG \nonumber \\
&=(\mu- \nu)^2 D_{0,j} + (\lambda - \nu)kG, \label{lambdanu}
\end{align}
which gives \eqref{Cond1} provided the relation
\[
(\lambda- \nu)k = \nu( \mu - \nu).
\]
holds. This relation follows by multiplying \eqref{linkedprop} by $\mu-\nu$, subtracting \eqref{lambdanu}, and then counting terms on both sides.

\end{proof}

By Proposition~\ref{straightforward}, the values of the parameters $\mu, \nu$ for a (non-reduced) linking system of difference sets are also as stated in Lemma~\ref{MuNuReduced}. These values were noted in \cite[p.~94]{davis-martin-polhill} (with a typographical error switching their values) as following from~\cite[Proposition~0]{noda}.

\section{Bent sets}
\label{sec:Boolean}
Bey and Kyureghyan \cite{bey-kyureghyan}, building on earlier work of Cameron and Seidel \cite{cameron-seidel}, Delsarte \cite{delsarte-association}, and Noda \cite{noda}, provide the second of the two principal references on linking systems of difference sets. Their main result is phrased in terms of systems of linked symmetric designs, rather than linking systems of difference sets (which were not defined until 2014 in \cite{davis-martin-polhill}). 
In this section we rephrase the main result of \cite{bey-kyureghyan} in terms of the newer terminology to give Corollary~\ref{bent=link} and to clarify some of the relationships to previous work. 
Corollary~\ref{bent=link} partially answers Q1 of Section \ref{overview} by showing how the construction of \cite[Example~6.2]{davis-martin-polhill} can be improved using a Kerdock set.

A \emph{Boolean function on $\Z_2^n$} is a function $f$ from $\Z_2^n$ to $\Z_2$. The \emph{subset} of $\Z_2^n= \langle x_1, x_2, \dots, x_n \rangle$ \emph{corresponding to a Boolean function $f$ on $\Z_2^n$} is 
\[
S(f) = \{ x_1^{y_1} x_2^{y_2} \cdots  x_n^{y_n} : f(y_1,y_2, \dots, y_n) =1 \}.
\]
The \emph{Walsh-Hadamard transform} of a Boolean function $f$ on $\Z_2^n$ is the function $\widehat{f}: \Z_2^n \rightarrow \Z$ given by 
\begin{equation*}
\widehat{f}(u)= \sum_{x \in \Z_2^n} (-1)^{f(x) +u \cdot x} \quad \text{ for } u \in \Z_2^n, 
\end{equation*}
where $\cdot$ is the usual inner product on $\Z_2^n$.

\begin{definition}
A Boolean function $f$ on $\Z_2^n$ is \emph{bent} if 
\begin{equation*}
\widehat{f}(u) \in \{2^{n/2}, - 2^{n/2}\} \quad \text{ for all } u \in \Z_2^n.
\end{equation*}
\end{definition}

Bent functions are closely connected to difference sets in elementary abelian 2-groups, as shown in the following result. 
\begin{theorem}[\cite{dillon-phd}]
A Boolean function $f$ on $\Z_2^{2d+2}$ is bent if and only if $S(f)$ is a difference set in $\Z_2^{2d+2}$ . 
\end{theorem}

\begin{definition}
A \emph{bent set on $\Z_2^{2n}$} of size $\ell+1$ is a set $\{f_0,f_1, \dots, f_\ell \}$ of Boolean functions on $\Z_2^{2n}$ such that the Boolean function $f_i+f_j$ is bent for all distinct $i,j$.
\end{definition}

We may assume (by adding one function to all the others) that one function in a bent set is the zero function. We now state the main result of \cite{bey-kyureghyan}, rephrased in terms of linking systems of difference sets. 

\begin{theorem}[Bey and Kyureghyan {\cite[Theorem~1]{bey-kyureghyan}}]\label{bent1}
Let $\ell \ge 2$ and suppose $\{0,f_1, \dots, f_\ell\}$ is a bent set on $\Z_2^{2d+2}$. Then $\{S(f_1), S(f_2), \dots, S(f_\ell)\}$ is a reduced linking system of difference sets in $\Z_2^{2d+2}$.
\end{theorem}

The following result describes a well-known construction of a bent set on $\Z_2^{2d+2}$, due originally to Kerdock \cite{kerdock}.

\begin{theorem}[\cite{kerdock}, {\cite[page~456]{macwilliams-sloane-book}}]\label{Kerdock}
For each integer $d\ge 0$, there exists a bent set  on $\Z_2^{2d+2}$ of size $2^{2d+1}$. 
\end{theorem}

We remark that Cameron and Seidel \cite{cameron-seidel} used a Kerdock set to construct a system of linked symmetric designs, and that Theorem~\ref{bent1} generalizes their construction by replacing a Kerdock set with a bent set.
Combining Theorems \ref{bent1} and \ref{Kerdock}, we obtain the following corollary. Note that the parameters of the difference sets in Corollary \ref{bent=link} are determined by Theorem \ref{2grp}.

\begin{corollary}\label{bent=link}
For each integer $d \ge 1$, there exists a reduced linking system of difference sets in $\Z_2^{2d+2}$ of size $2^{2d+1}-1$.
\end{corollary}

We refer to \cite[Example 1.38]{simon-masters} for an example of a bent set on $\Z_2^4$ of size 8 and the corresponding reduced $(16,6,2,4;7)$-linking system of difference sets in $\Z_2^4$.

We cannot use Theorem \ref{bent1} to produce a reduced system of linking difference sets larger than that in Corollary \ref{bent=link}, because the bent sets of Theorem \ref{Kerdock} attain the maximum size by the following result. 

\begin{theorem}[Delsarte {\cite[p.~82]{delsarte-association}}, Bey and Kyureghyan {\cite[Theorem~2]{bey-kyureghyan}}]\label{bent2}
For each integer $d \ge 0$, there is no bent set on $\Z_2^{2d+2}$ of size greater than $2^{2d+1}$.
\end{theorem}

\section{Nonexistence results in non-2-groups}
\label{sec:nonexistence}
Several authors have established constraints on the existence of systems of linked symmetric designs \cite{cameron-doubly}, \cite{noda}, \cite{mathon}, \cite{kodalen}, which in turn imply the nonexistence of corresponding linking systems of difference sets. In particular, Kodalen \cite[Appendix~1]{kodalen} determined which of the 21 known families of symmetric designs have parameters that satisfy integrality conditions (corresponding to those arising from Lemma~\ref{MuNuReduced}) necessary for a system of linked symmetric designs to exist.
However, nonexistence results that apply only to linking systems of difference sets have not previously been found.
In this section we uncover an obstruction to the existence of a reduced linking system of difference sets in non-$2$-groups, both when the difference sets are constructed by the McFarland/Dillon method (Theorem \ref{McFarland}) and when they are constructed by the Spence method (Theorem \ref{Spence}). The results are not restricted to abelian groups. 

Our nonexistence proofs use only elementary arguments, combining properties of hyperplanes given in Proposition~\ref{HYPE}, modular reduction in the group ring, and projection to a subgroup. Each of these techniques is well-known; the novelty of the proof lies in recognizing the correct modular reduction, and in combining the various ingredients in the correct order.

\subsection{McFarland/Dillon and Spence constructions}

We first present the constructions originally given by McFarland (and later modified by Dillon) and Spence for the parameter families named after them. Both constructions rely on the properties of hyperplanes of a vector space.

\begin{definition}
\label{def:hyperplane}
Let $V$ be a vector space of dimension $d+1$ over $\GF$$(q)$. A \emph{hyperplane} of $V$ is a $d$-dimensional subspace of $V$.
\end{definition}
The number of hyperplanes in the vector space $V$ of Definition~\ref{def:hyperplane} is $\frac{q^{d+1}-1}{q-1}$.

\begin{theorem}[McFarland \cite{mcfarland-noncyclic}, Dillon \cite{dillon-variations}] \label{McFarland}
Let $q$ be a prime power and $d$ a nonnegative integer, and let $s = \frac{q^{d+1}-1}{q-1}$. Let $G$ be a group containing a central subgroup $E$ of index $s+1$ isomorphic to the elementary abelian group of order $q^{d+1}$.  Let $g_0, g_1, \dots , g_s$ be a set of coset representatives for $E$ in $G$. Let $H_1, H_2,\dots , H_s$ be the subgroups of $G$ corresponding to the hyperplanes of $E$, under an isomorphism $\phi$, when $E$ is regarded as a vector space of dimension $d+1$ over $\GF$$(q)$. Then
\begin{equation*}
D=\sum_{i=1}^s g_i H_i
\end{equation*}
is a difference set in $G$ with McFarland parameters
$(v,k,\lambda,n) = (q^{d+1}(s+1), q^ds, q^d (s-q^d),q^{2d})$.
\end{theorem}

\begin{theorem}[Spence \cite{spence}] \label{Spence}
Let $d\ge 0$ and let $s = \frac{3^{d+1}-1}{2}$. Let $G$ be a group containing a central subgroup $E$ of index $s$ isomorphic to~$\Z_3^{d+1}$.  Let $ g_1, \dots , g_s$ be a set of coset representatives for $E$ in $G$. Let $H_1, H_2, \dots , H_s$ be the subgroups of $G$ corresponding to the hyperplanes of $E$ when $E$ is regarded as a vector space of dimension $d+1$ over $\GF$$(3)$. Then
\begin{equation*}
D=g_1(E-H_1 )+ \sum_{i=2}^s g_i H_i
\end{equation*}
is a difference in $G$ with Spence parameters
$(v,k,\lambda,n) = \left( 3^{d+1} s, 3^d (s+1), 3^d (s+1-3^d), 3^{2d} \right)$.
\end{theorem}

In the McFarland/Dillon construction of Theorem~\ref{McFarland}: 
the subgroup $E$ has index $s+1$ in $G$; 
the subgroups $H_1, H_2, \dots, H_s$ of $G$ corresponding to the hyperplanes of $E$ depend on the isomorphism $\phi$ when $q$ is not a prime; 
and the difference set $D$ comprises one coset of each of these $s$ subgroups.
In contrast, in the Spence construction of Theorem~\ref{Spence}: 
the subgroup $E$ has index $s$ in $G$; 
the subgroups $H_1, H_2, \dots, H_s$ of $G$ corresponding to the hyperplanes of $E$ are determined without reference to an isomorphism $\phi$, because the construction is over the prime field \GF(3);
and the difference set $D$ comprises a coset of the complement in $E$ of one of the subgroups together with a coset of each of the remaining $s-1$ subgroups. 
Note that constructions other than Theorem \ref{McFarland} are known for difference sets with McFarland parameters~\cite{unify}.

We now state our two nonexistence results.

\begin{theorem}\label{NoMcFarland}
Let $ q > 2$ be a prime power and $d$ a positive integer, and let $s = \frac{q^{d+1}-1}{q-1}$. Then there is no reduced linking system of difference sets with McFarland parameters 
\begin{equation*}
(v,k,\lambda, n)=  \left( q^{d+1} \left(s+1 \right), q^d s, q^d (s-q^d ), q^{2d} \right)
\end{equation*}
in which two of the difference sets are constructed as in Theorem~\ref{McFarland} with respect to the same subgroup $E$ and the same isomorphism~$\phi$. 
\end{theorem}

\begin{theorem}\label{NoSpence}
Let $d$ be a positive integer and let $s = \frac{3^{d+1}-1}{2}$. Then there is no reduced linking system of difference sets with Spence parameters 
\begin{equation*}
(v,k, \lambda, n) = \left( 3^{d+1} s, 3^d (s+1), 3^d (s+1-3^d), 3^{2d} \right)
\end{equation*}
in which two of the difference sets are constructed as in Theorem \ref{Spence}.
\end{theorem}

The condition in Theorem~\ref{NoMcFarland}, that the two difference sets are constructed with respect to the same subgroup $E$, can be omitted when $q = p^r$ for an odd prime $p$: the central subgroup $E$ is then a Sylow $p$-subgroup of the group $G$ of order $q^{d+1}(s+1)$ because $\gcd(p,s+1)=1$, and so is unique by Sylow's Third Theorem. Similarly, this condition is not needed in Theorem~\ref{NoSpence}. The condition in Theorem~\ref{NoMcFarland}, that the two difference sets are constructed with respect to the same isomorphism $\phi$, can be omitted when $q$ is a prime.

By imposing the condition that $d$ should be a positive integer in Theorems \ref{NoMcFarland} and \ref{NoSpence}, we exclude trivial McFarland and Spence difference sets containing a single element. We can obtain a stronger nonexistence result than Theorem~\ref{NoSpence} in the case $d=1$ when the group is abelian, because the classification result given by Turyn {\cite[Theorem~10]{turyn-charsums}} and completed by Spence  {\cite[Section~2]{spence}} states that every $(36,15,6,9)$-difference set in an abelian group of order 36 is constructed as in Theorem \ref{Spence} (for some labelling of the subgroups $H_1,H_2, H_3,H_4$).

\begin{corollary}\label{cor:36}
There is no reduced linking system of $(36,15,6,9)$-difference sets in an abelian group. 
\end{corollary}

The parameters of a $(v,k,\lambda,n)$-difference set in $\Z_2^2 \times \Z_3^2$ or $\Z_4 \times \Z_3^2$ satisfying $2 \le k \le v/2$ must be $(36, 15, 6, 9)$, by solving the difference set counting relation $k(k-1)=\lambda(v-1)$ for $v=36$.
Therefore by Corollary~\ref{cor:36}, each of $\Z_2^2 \times \Z_3^2$ and $\Z_4 \times \Z_3^2$ is an abelian group $G$ for which it is known that a $(v,k,\lambda,n)$-difference set exists in $G$ and the corresponding values of $\mu$ and $\nu$ specified in Lemma~\ref{MuNuReduced} are integers, but a (reduced) linking system of difference sets does not exist in~$G$. The only other such abelian group we are aware of is $\Z_8 \times \Z_2$, for which we determined the nonexistence result using Theorem~\ref{2grp} and exhaustive search. 

\subsection{Proof of Theorems \ref{NoMcFarland} and \ref{NoSpence}}\label{nonexistproof}

We next derive some divisibility conditions (Lemmas~\ref{determined} and~\ref{determined2}), and state a well-known result on hyperplanes (Proposition~\ref{HYPE}).

\begin{lemma}\label{determined}
Let $q >2$ be a prime power and $d$ a positive integer, and let $s = \frac{q^{d+1}-1}{q-1}$. Then $s+1$ does not divide $q^{d-1}s(s-1)$. 
\end{lemma}

\begin{proof}
Suppose, for a contradiction, that $s+1$ divides $q^{d-1}s(s-1)$. Since $\gcd(s+1,s)=1$, this implies that 
\begin{equation}\label{divides}
s+1 \text{ divides } q^{d-1}(s-1). 
\end{equation}
Note that 
\begin{align}\label{s+1}
s+1=2+q+q^2+ \cdots +q^d.
\end{align}
\begin{description}
\item[Case 1: $q$ is odd.] We have $\gcd(s+1,q)=1$ from \eqref{s+1} and then \eqref{divides} implies that $s+1$ divides $s-1$. This is a contradiction because $s > 1$.
\item[Case 2: $q > 2$ is a power of 2.] 
Then $s$ is odd and so $\gcd(s+1,s-1)=2$. Then \eqref{divides} implies that $s+1$ divides $2q^{d-1}$. This is a contradiction because $s+1$ is not a power of 2, by \eqref{s+1}, whereas $2q^{d-1}$ is a power of 2. 
\end{description}
\end{proof}

\begin{lemma}\label{determined2}
Let $d$ be a positive integer and let $s = \frac{3^{d+1}-1}{2}$. Then $s$ does not divide $3^{d-1}(s+1)(s+2)$. 
\end{lemma}

\begin{proof}
Suppose, for a contradiction, that $s$ divides $3^{d-1}(s+1)(s+2)$. Since $\gcd(s,s+1)=1$, this implies that $s$ divides $3^{d-1}(s+2)$. Writing $s=1+3+3^2+ \cdots + 3^d$ shows that $\gcd(s,3)=1$ and we therefore deduce that $s$ divides $s+2$. This is a contradiction because $s>2$. 
\end{proof}

\begin{proposition}[\cite{mcfarland-noncyclic}]\label{HYPE}
Let $H_i$ and  $H_j$ be hyperplanes of a vector space $V$ of dimension $d+1$ over $\GF$$(q)$. Then in the group ring $\Z[V]$,
\begin{align*}
H_iH_j=
\begin{cases} 
      \hfill q^{d} H_i   \hfill & \text{ if } H_i = H_j, \\
      \hfill q^{d-1} V \hfill &\text{ if } H_i \ne H_j. \\
  \end{cases}
  \end{align*}
\end{proposition}

We can now prove Theorems \ref{NoMcFarland} and~\ref{NoSpence}.

\begin{proof}[Proof of Theorem \ref{NoMcFarland}]
Let $G$ be a group containing a central subgroup $E$ of index $s+1$ isomorphic to the elementary abelian group of order $q^{d+1}$. Let $f_0, f_1, \dots , f_s$ and $g_0, g_1, \dots, g_s$ each be a set of coset representatives for $E$ in $G$. Let $H_1, H_2, \dots, H_s$ be the subgroups of $G$ corresponding to the hyperplanes of $E$, under an isomorphism~$\phi$, when $E$ is regarded as a vector space of dimension $d+1$ over $\GF$$(q)$ and let 
\begin{align*}
D_1= \sum_{i=1}^s f_i H_i \quad \text{ and }  \quad D_2 = \sum_{i=1}^s g_i H_i.
\end{align*}
Suppose, for a contradiction, that there are integers $\mu,\nu$ such that 
\begin{align}\label{eq1}
D_1D_2^{(-1)}= (\mu-\nu) D+ \nu G  \quad \text{ in } \Z[G]
\end{align}
for some difference set $D$ (having the same parameters $(v,k,\lambda,n)$ as $D_1, D_2$) in~$G$. By Lemma~\ref{MuNuReduced}, 
\begin{align}\label{McMuNu}
\nu = q^{d-1} \frac{s (s \pm 1)}{s+1} \quad \text{ and} \quad \mu = \nu \mp q^d.
\end{align}
By Lemma \ref{determined} we cannot take the lower signs in \eqref{McMuNu}, and so \eqref{eq1} becomes
\begin{equation}\label{pluggedin}
D_1D_2^{(-1)}= -q^d D+ q^{d-1} s G.
\end{equation}

Now, $E$ is a central subgroup containing each $H_i$, and $H_i = H_i^{(-1)}$, so 
\begin{align*}
D_1D_2^{(-1)} &= \sum_{i=1}^s f_i H_i \sum_{j=1}^s(g_jH_j)^{(-1)}\\
&= \sum_{1 \le i,j \le s} f_i g_j^{-1}H_iH_j \\
&= \sum_{i=1}^s f_ig_i^{-1} (q^d H_i) + \sum_{\stackrel{1\le i,j \le s}{i \ne j}} f_i g_j^{-1} (q^{d-1} E),
\end{align*}
by separating into sums over $i=j$ and $i \ne j$, and using Proposition~\ref{HYPE}. Substitute into \eqref{pluggedin} and reduce modulo $q^d$ to obtain
\begin{align*}
q^{d-1} \sum_{\stackrel{1\le i,j \le s}{i \ne j}} f_i g_j^{-1} E \equiv q^{d-1}sG \pmod{q^d} \quad \text{ in } \Z[G].
\end{align*}
Therefore
\begin{align}\label{countfromhere}
 \sum_{\stackrel{1\le i,j \le s}{i \ne j}} f_i g_j^{-1} E \equiv sG \pmod{q} \quad \text{ in } \Z[G].
\end{align}

Let $K = \{k_0, k_1, \dots, k_s \}$ be a set of coset representatives for $E$ in $G$. Each $g \in G$ may be uniquely represented as $k_t e$ for some $k_t \in K$ and some $e \in E$, from which we define a projection map $\rho: G \rightarrow E$ by
\begin{align*}
\rho(k_t e) = e \quad \text{ for } k_t \in K \text{ and } e \in E. 
\end{align*}
The map $\rho$ induces a projection from $\Z[G]$ to $\Z[E]$.
 For each distinct $i,j$, write $f_i g_j^{-1} \in G$ uniquely as 
 \begin{align*}
 f_i g_j^{-1} = k_{t(i,j)} e_{i,j} \quad \text{ where } k_{t(i,j)} \in K \text{ and } e_{i,j} \in E,
 \end{align*}
 and write $G = \sum_{t=0}^s \sum_{e \in E} k_t e$ so that \eqref{countfromhere} becomes
\begin{align}
 \sum_{\stackrel{1\le i,j \le s}{i \ne j}} \sum_{e \in E} k_{t(i,j)} e_{i,j} e  \equiv s \sum_{t=0}^s  \sum_{e \in E} k_t e \pmod{q} \quad \text{ in } \Z[G].
\end{align}
Apply $\rho$ to both sides to give 
\begin{align}
 \sum_{\stackrel{1\le i,j \le s}{i \ne j}} \sum_{e \in E} e_{i,j} e  \equiv s \sum_{t=0}^s  \sum_{e \in E} e \pmod{q} \quad \text{ in } \Z[E].
\end{align}
Using $\sum_{e\in E} e_{i,j} e= e_{i,j}E=E$ then gives 
\begin{align*}
s(s-1)E \equiv s(s+1)E \pmod{q} \quad \text{ in } \Z[E].
\end{align*}
Compare the coefficient of $\bone_E$ on both sides to give 
\begin{align*}
s(s-1) \equiv s(s+1) \pmod{q}.
\end{align*}
Since $s=1+q+q^2+ \cdots + q^d$, this implies 
\begin{align*}
0 \equiv 2 \pmod{q},
\end{align*}
which is a contradiction because $q>2$. 
\end{proof}

\begin{proof}[Proof of Theorem \ref{NoSpence}]
The proof is similar to that of Theorem \ref{NoMcFarland}. We use the same strategy of expanding the product $D_1D_2^{(-1)}$, taking a modular reduction, and taking a projection map. We highlight the places in which additional care is needed.

Let $G$ be a group containing a central subgroup $E$ of index $s$ isomorphic to $\Z_3^{d+1}$. Let $f_1, \dots , f_s$ and $g_1, \dots, g_s$ each be a set of coset representatives for $E$ in $G$. Let $H_1, H_2, \dots, H_s$ be the subgroups of $G$ corresponding to the hyperplanes of $E$ when $E$ is regarded as a vector space of dimension $d+1$ over $\GF$$(3)$ and let 
\begin{align*}
D_1=f_1(E-H_1)+ \sum_{i\ne 1} f_i H_i, \quad \text{ and } \quad D_2 =g_m(E-H_m)+ \sum_{j \ne m} g_j H_j,
\end{align*}
where $D_1$ involves the complement in $E$ of subgroup $H_1$, and $D_2$ involves the complement in $E$ of subgroup $H_m$: we must examine both the cases $m=1$ and $m \ne 1$. 
Suppose, for a contradiction, that there are integers $\mu,\nu$ such that 
\begin{align}\label{eq1SECOND}
D_1D_2^{(-1)}= (\mu-\nu) D+ \nu G  \quad \text{ in } \Z[G]
\end{align}
for some difference set $D$ (having the same parameters $(v,k,\lambda,n)$ as $D_1, D_2$) in~$G$. 
By Lemma~\ref{MuNuReduced}, 
\begin{align}\label{SpenceMuNu}
\nu = 3^{d-1} \frac{(s+1) (s+1 \pm 1)}{s} \quad \text{ and} \quad \mu = \nu \mp 3^d.
\end{align}
By Lemma \ref{determined2} we cannot take the upper signs in \eqref{SpenceMuNu}, and so \eqref{eq1SECOND} becomes
\begin{equation}\label{pluggedinSpence}
D_1D_2^{(-1)}= 3^d D+ 3^{d-1} (s+1) G.
\end{equation}
Substitute for $D_1$ and $D_2$, and reduce modulo $3^d$ to give 
\begin{align}\label{SpenceStar}
\left( f_1(E-H_1) + \sum_{i \ne 1} f_i H_i \right) \left( g_m^{-1}(E-H_m) + \sum_{j \ne m} g_j^{-1} H_j \right) 
\equiv 3^{d-1} (s+1)G \pmod{3^d} \text{ in } \Z[G].
\end{align}
By Proposition~\ref{HYPE}, 
\begin{align*}
H_i E \equiv EE \equiv H_i H_i \equiv 0 \pmod{3^d} \text{ in } \Z[G]
\end{align*}
and so we need retain on the left hand side of \eqref{SpenceStar} only those terms involving $H_iH_j$ for distinct~$i,j$. 

\begin{description}
\item[Case 1: $m=1$.] 
By Proposition~\ref{HYPE}, the congruence \eqref{SpenceStar} becomes 
\[
 -3^{d-1} \sum_{j \ne 1} f_1 g_j^{-1} E -3^{d-1} \sum_{i \ne 1} f_i g_1^{-1} E+ 3^{d-1} \sum_{\stackrel{2\le i,j \le s}{i \ne j}} f_i g_j^{-1} E 
 \equiv 3^{d-1}(s+1) G  \pmod{3^d} \text{ in } \Z[G]. 
\]
Applying a projection map $\rho$ from $G$ to $E$ as in the proof of Theorem \ref{NoMcFarland}, we deduce that 
\begin{align*}
-(s-1)E-(s-1)E+(s-1)(s-2)E \equiv (s+1)sE  \pmod{3} \text{ in } \Z[E]. 
\end{align*}
Since $s \equiv 1 \pmod{3}$, this gives the contradiction 
\begin{align*}
0 \equiv 2 \pmod{3}.
\end{align*}

\item[Case 2: $m \ne 1$.]
By Proposition~\ref{HYPE}, the congruence \eqref{SpenceStar} becomes 
\begin{align*}
& 3^{d-1} f_1 g_m^{-1}E- 3^{d-1} \sum_{j \ne 1,m} f_1 g_j^{-1}E-3^{d-1} \sum_{i \ne 1,m} f_i g_m^{-1}E + 
3^{d-1} \sum_{\stackrel{1\le i,j \le s}{i \ne 1, j \ne m, i \ne j}} f_i g_j^{-1} E \\
&\hspace{85mm} \equiv 3^{d-1}(s+1) G  \pmod{3^d} \text{ in } \Z[G],
\end{align*}
which after projection gives 
\begin{align*}
E-(s-2)E-(s-2)E+\left((s-1)^2-(s-2) \right)E \equiv (s+1)s E \pmod{3}  \text{ in } \Z[E].
\end{align*}
Since $s \equiv 1 \pmod{3}$, this gives the contradiction 
\begin{align*}
1 \equiv 2 \pmod{3}.
\end{align*}
\end{description}
\end{proof}

\section{Constructions in 2-groups using group difference matrices}
\label{sec:construction}

In this section we present a powerful construction (Theorem \ref{Main}) of reduced linking systems of difference sets in $2$-groups. The construction, which is not restricted to abelian groups,  combines combinatorial properties of hyperplanes (Proposition~\ref{HYPE}) with the unexpected use of group difference matrices. 

Theorem~\ref{Main} has several consequences.
Corollaries~\ref{general} and~\ref{improved} construct infinite families of reduced linking systems of difference sets in abelian groups, simplifying and extending some of the previous results given in Theorem~\ref{thm-dmp}.
Corollary~\ref{TYKEN} constructs an infinite family of examples in nonabelian groups, whereas not a single nonabelian example was previously known. 
Corollary~\ref{cor-nonrev} constructs an infinite family of nonreversible examples generalizing Example~\ref{FromDMP}. 
Finally, we show in Section \ref{subsec-Z42} that the construction produces all possible examples of maximum size in the group $\Z_4^2$, and allows significant control over which difference sets in the reduced linking system are reversible. 
\subsection{Group difference matrices}\label{DMs}

We first introduce group difference matrices; see \cite{colbourn-diffmatrices} for a survey of the topic.

\begin{definition}\label{DMdef}
Let $G$ be a group of order $v>1$. A \emph{$(G, m, \lambda)$-difference matrix} is an $m \times \lambda v$ matrix $(b_{i,j})$ with $0 \le i \le m-1$ and $0 \le j \le \lambda v-1$ and each entry $b_{i,j} \in G$ such that, 
for all distinct rows $i$ and $r$, the multiset $\{ b_{i,j} b_{r,j}^{-1} : 0 \le j \le \lambda v-1 \}$ contains every element of $G$ exactly $\lambda$ times.
\end{definition}

We shall be interested only in the case $\lambda=1$ of Definition \ref{DMdef}, so that for each distinct $i,r$ the set $\{b_{i,j} b_{r,j}^{-1} : 0 \le j \le v-1 \}$ contains every element of $G$ exactly once. We can right-multiply all entries of a column of a $(G,m,1)$-difference matrix by a fixed $a \in G$ without changing the defining property of the matrix, because $(b_{i,j}a)(b_{r,j}a)^{-1}= b_{i,j} b_{r,j}^{-1}$. By right-multiplying all entries of each column $j$ by $b_{0,j}^{-1}$, we may therefore assume that each entry of row $0$ of the matrix is $\bone_G$. The difference property of the matrix then implies that, for each $i \ge 1$, the set 
$\{ b_{i,j} : 0 \le j \le v-1\}$ contains every element of $G$ exactly once. We can likewise right-multiply all entries of each row by $b_{i,0}^{-1}$, so that each entry of column 0 of the matrix is $\bone_G$. By considering the entries of column 1 of the matrix and using the pigeonhole principle, we see that the largest number of rows of a $(G,m,1)$-difference matrix is~$\vert G \vert$. 

\begin{example}\label{DMEX}
Let $G=\Z_2^2 = \langle x,y \rangle$. The matrix
\begin{align*} 
(b_{i,j}) = 
\begin{pmatrix}
\bone_G & \bone_G	& \bone_G	& \bone_G \\
\bone_G & x		& y 		& xy \\
\bone_G & y 		& xy 		& x \\
\bone_G & xy 		& x 		& y
\end{pmatrix}
\end{align*}
 is a $(\Z_2^2,4,1)$-difference matrix. 
\end{example}

We shall make use of the following two constructive results for difference matrices in abelian $2$-groups. 

\begin{theorem}[Pan and Chang {\cite[Lemma~3.4]{pan-chang}}]\label{generalDM}
Let $G$ be an abelian noncyclic $2$-group. Then there exists a $(G,4,1)$-difference matrix. 
\end{theorem}

\begin{theorem}[Buratti {\cite[Theorem~2.11]{buratti}}]\label{Buratti}
Let $G$ be an abelian group of order $2^{d+1}$ and exponent $2^e$. Then there exists a 
$(G,2^{\left\lfloor \frac{d+1}{e} \right\rfloor},1)$-difference matrix. 
\end{theorem}
Note that the case $e=1$ of Theorem \ref{Buratti} gives a $(G,m,1)$-difference matrix for $G= \Z_2^{d+1}$ and $m=2^{d+1}$, which satisfies the extremal condition $m=\vert G \vert$.

\subsection{Main construction theorem}\label{subsecmain}
We next find sufficient conditions for the linking property \eqref{LinkingProperty} to hold for a pair of difference sets in a $2$-group (Lemma \ref{newconstruction}), and then show that the rows of a difference matrix can be used to satisfy these pairwise conditions for a collection of difference sets simultaneously (Theorem \ref{Main}). 
Recall that the parameters of a difference set in a group of order $2^{2d+2}$ are determined by Theorem~\ref{2grp}, and can be regarded as McFarland parameters with $q=2$, and that when $q$ is prime we do not need to specify an isomorphism $\phi$ in Theorem~\ref{McFarland}.

\begin{lemma}\label{newconstruction}
Let $d$ be a nonnegative integer and let $s = 2^{d+1}-1$. Let $G$ be a group of order $2^{2d+2}$ which contains a central subgroup $E$ isomorphic to $\Z_2^{d+1}$. Suppose that $f_0, f_1, \dots f_s$ and $g_0, g_1, \dots, g_{s}$ are each a set of coset representatives for $E$ in $G$ such that 
$f_0g_0^{-1}, f_1g_1^{-1}, \dots, f_sg_s^{-1}$ is also a set of coset representatives for $E$ in~$G$. 
Let $H_1, H_2, \dots, H_s$ be the subgroups of $G$ corresponding to the hyperplanes of $E$ when $E$ is regarded as a vector space of dimension $d+1$ over $\GF$$(2)$. 
Then the sets
\begin{align*}
D_1 = \sum_{i=1}^s f_i H_i \quad \text{ and } \quad 
D_2 = \sum_{i=1}^s g_i H_i
\end{align*}
are difference sets in $G$ satisfying
\begin{equation*}
D_1D_2^{(-1)}= -2^dD+ 2^{d-1}s G \quad \text{ in } \Z[G],
\end{equation*}
where 
\begin{align}\label{EasyD}
D=\sum_{i=1}^s f_i g_i^{-1} (E-H_i)
\end{align}
is also a difference set in~$G$. 
\end{lemma}

\begin{proof}
The sets $D_1$ and $ D_2$ are difference sets in $G$, by Theorem \ref{McFarland} with $q=2$.
Since each $H_j \subset E$ is central in $G$, and $H_j^{(-1)} = H_j$, by Proposition~\ref{HYPE} we calculate in $\Z[G]$ that
\begin{align}\label{step1}
 D_1D_2^{(-1)}=\left( \sum_{i=1}^s f_i H_i \right)  \left( \sum_{j=1}^s H_j^{(-1)}g_{j}^{-1}  \right) &=
 \sum_{i=1}^s f_i g_{i}^{-1} (2^{d} H_i) + \sum_{\stackrel{1\le i,j \le s}{i \ne j}} f_i g_{j}^{-1} (2^{d-1} E).
\end{align}
By assumption, each of $\{f_i : 0 \le i \le s\}$ and $\{g_i : 0 \le i \le s\}$ and $\{f_i g_i^{-1} : 0 \le i \le s\}$ is a set of coset representatives for $E$ in $G$ and so 
\begin{align}
\sum_{i=1}^s f_i E &= G-f_0E \nonumber \\
\sum_{i=1}^s g_i^{-1} E &= G-g_0^{-1}E \nonumber \\
\sum_{i=1}^s f_ig_i^{-1} E &= G-f_0g_0^{-1}E. \label{G-E}
\end{align}
Therefore
\begin{align*}
\sum_{\stackrel{1\le i,j \le s}{i \ne j}} f_i g_{j}^{-1}E
 &= \sum_{i=1}^sf_i \sum_{j=1}^s g_j^{-1}E-\sum_{i=1}^s f_i g_i^{-1}E \\
 &= \sum_{i=1}^s f_i(G-g_0^{-1}E) - (G-f_0g_0^{-1}E) \\
 &=sG- (G-f_0E)g_0^{-1}-(G-f_0 g_0^{-1}E )\\
 &=(s-2)G+2f_0g_0^{-1}E.
\end{align*}
Substitute into \eqref{step1} to obtain
\begin{align}
D_1D_2^{(-1)} 
 &=2^d \sum_{i=1}^s f_i g_i^{-1} H_i+ 2^{d-1} \left( (s-2)G+2f_0g_0^{-1}E \right) \label{combinewhen2}\\
 &=-2^d \left(G - \sum_{i=1}^s f_i g_i^{-1}H_i - f_0 g_0^{-1}E \right) + 2^{d-1}sG \nonumber \\
 &=-2^d D + 2^{d-1} sG, \label{2terms}
\end{align}
where by \eqref{G-E} 
\begin{align*}
D=\sum_{i=1}^s f_i g_i^{-1} (E-H_i).
\end{align*}

It remains to show that $D$ is a difference set in $G$. 
For each $i$ we may write $E-H_i=a_i H_i$ for some $a_i \in E$ because $H_i$ has index $2$ in $E \cong \Z_2^{d+1}$, and then
$
D= \sum_{i=1}^s f_i g_i^{-1} a_i H_i.
$
Since  $\{f_i g_i^{-1} : 0 \le i \le s\}$ is a set of coset representatives for $E$ in $G$, so is $\{f_i g_i^{-1}a_i : 0 \le i \le s \}$ and therefore $D$ is a difference set in $G$ by Theorem \ref{McFarland}. 
\end{proof}

Lemma \ref{newconstruction} shows that, subject to conditions on coset representatives, McFarland difference sets $D_1,D_2$ with $q=2$ that are constructed as in Theorem \ref{McFarland} can form a reduced linking system. 
In view of the nonexistence result of Theorem \ref{NoMcFarland} (also for McFarland difference sets with $q=2$ that are constructed as in Theorem~\ref{McFarland}), it is natural to ask where Lemma \ref{newconstruction} fails for $q>2$. There are two such places. Firstly, the expression 
\begin{align*}
2^d \sum_{i=1}^s f_i g_i^{-1} H_i+ 2^{d-1} \left( (s-2)G+2f_0g_0^{-1}E \right)
\end{align*}
in \eqref{combinewhen2} is replaced by
 \begin{align*}
q^d \sum_{i=1}^s f_i g_i^{-1} H_i+q^{d-1} \left( (s-2)G+2f_0g_0^{-1}E \right),
\end{align*}
so that the corresponding expression \eqref{2terms} for $D_1D_2^{(-1)}$ now involves three distinct coefficients $q^d,q^{d-1}s,2q^{d-1}$ (where $s = \frac{q^{d+1}-1}{q-1}$).
Secondly, it is no longer possible to write $E-H_i = a_iH_i$ for some $a_i \in E$, because $H_i$ has now index $q$ in $E$ (where $E$ is isomorphic to the elementary abelian group of order $q^{d+1}$) .

\begin{theorem}\label{Main}
Let $G$ be a group of order $2^{2d+2}$ which contains a central subgroup $E$ isomorphic to $\Z_2^{d+1}$. Let $m \ge 3$ and suppose there exists a  $(G/E, m,1)$-difference matrix. Then $G$ contains a reduced linking system of difference sets of size $m-1$.
\end{theorem}

\begin{proof}
Let $s= 2^{d+1}-1$ and let $H_1,H_2, \dots, H_s$ be the subgroups of $G$ corresponding to the hyperplanes of $E$ when $E$ is regarded as a vector space of dimension $d+1$ over $\GF$$(2)$.
Let the $(G/E,m,1)$-difference matrix be $B=(b_{i,j}E)$ for $0 \le i \le m-1$ and $0 \le j \le s$ and $b_{i,j} \in G$. As noted after Definition~\ref{DMdef} we may assume that, for each nonzero distinct $i$ and $r$, the set 
$\{b_{i,j}E : 0 \le j \le s \}$, as well as the set $\{ b_{i,j}b_{r,j}^{-1}E : 0 \le j \le s \}$, contains every element of $G/E$ exactly once. Therefore, the sets $\{b_{i,j}: 0 \le j \le s \}$ and 
$\{b_{i,j} b_{r,j}^{-1}: 0 \le j \le s \}$ are both a set of coset representatives for $E$ in $G$.
Choose $e_{i,j} \in E$ for each $1 \le i \le m-1$ and $ 1 \le j \le s$ arbitrarily. Since $E$ is central in $G$, it follows that the sets $\{b_{i,j}e_{i,j} :0 \le j \le s \}$ and $\{(b_{i,j} e_{i,j} )(b_{r,j}e_{r,j})^{-1} : 0 \le j \le s \}$ are both a set of coset representatives for $E$ in $G$.

Let
\begin{align}\label{DfromDM}
D_i = \sum_{j=1}^s b_{i,j} e_{i,j} H_j \quad \text{ for } 1 \le i \le m-1. 
\end{align}
We shall show that $\{D_1, D_2, \dots, D_{m-1}\}$ is a reduced linking system of difference sets in~$G$.
By Definition \ref{reduced}, we require that each $D_i$ is a difference set in $G$ and that there are integers $\mu, \nu$ such that, for all distinct nonzero $i$ and $r$,
\begin{equation*}
D_i D_r^{(-1)}= (\mu - \nu)D(i,r) + \nu G
\end{equation*}
for some difference set $D(i,r)$ in~$G$. 
This follows from Lemma \ref{newconstruction} by taking $f_j=b_{i,j}e_{i,j}$ and $g_j= b_{r,j}e_{r,j}$.
\end{proof}

We now give an example of the construction method of Lemma~\ref{newconstruction} and Theorem~\ref{Main} in which $G$ is not an elementary abelian $2$-group. 

\begin{example}
Let $G = \Z_{4} \times \Z_2 \times \Z_2 = \langle x,y,z \rangle$ and let $E = \langle x^2, z \rangle$. 
The subgroups of $G$ corresponding to the hyperplanes of $E$ when $E$ is regarded as a vector space of dimension $2$ over $\GF(2)$ are $H_1 = \langle x^2 \rangle, H_2 = \langle z \rangle, H_3 = \langle x^2 z \rangle$. By Example \ref{DMEX},
 the matrix $(b_{i,j}E)$ is a $(G/E,4,1)$-difference matrix, where
\begin{align*} 
(b_{i,j}) = 
\begin{pmatrix}
\bone_{G} 	& \bone_{G} 	& \bone_{G}	& \bone_{G} \\
\bone_{G}	& x		& y 		& xy \\
\bone_{G} 	& y 		& xy 		& x \\
\bone_{G} 	& xy 		& x 		& y
\end{pmatrix} \text{ for } 0 \le i,j \le 3. 
\end{align*}
Take
\begin{align*} 
(e_{i,j}) = 
\begin{pmatrix}
\bone_E	& \bone_E	& \bone_E \\
z 	& x^2 		& x^2 \\
z 	& \bone_E 	& \bone_E
\end{pmatrix} 
\text{ for } 1 \le i,j \le 3.
\end{align*}
Then \eqref{DfromDM} gives the reduced linking system of difference sets 
\begin{align*}
D_1 &=x  H_1 + y  H_2 + xy H_3 \\
D_2 &=y z H_1 + x^3y H_2 + x^3H_3 \\
D_3 &=xyz H_1 + x H_2 + y H_3.
\end{align*}
Furthermore, we can find the difference set $D(i,j)$ specified in Definition \ref{reduced} directly from $D_i$ and~$D_j$. For example,
\begin{align*}
D_2D_3^{(-1)}= -2D(2,3)+3G,
\end{align*}
where by \eqref{EasyD} we find
\begin{align*}
D(2,3) 
&=(yz)(xyz)^{-1}(E-H_1)+ (x^3y)(x)^{-1}(E-H_2) + (x^3)(y)^{-1}(E-H_3) \\
&=x^3(zH_1) +x^2y (x^2H_2)+ x^3y(zH_3)\\
&= x z H_1 + y H_2 + xy H_3.
\end{align*}

\end{example}

\subsection{Infinite families in abelian groups}\label{subsec-inf}

We now use Theorem \ref{Main} to construct infinite families of reduced linking systems of difference sets in a wide range of abelian $2$-groups. 

\begin{corollary}\label{general}
Let $G$ be an abelian group of order $2^{2d+2}$, rank at least $d+1$, and exponent at most~$2^{d+1}.$ Then $G$ contains a reduced linking system of difference sets of size~$3$.
\end{corollary}

\begin{proof}
Write $G = \Z_{2^{a_1}} \times \cdots \times \Z_{2^{a_{d+1+t}}}$, where $t \ge 0$ and $a_1 \ge a_2 \ge \cdots \ge a_{d+1+t} \ge 1$ and $\sum_i a_i = 2d+2$. The first $d+1$ direct factors of $G$ contain a subgroup $E$ isomorphic to $\Z_2^{d+1}$, and
\begin{equation*}
G/E \cong \Z_{2^{a_1-1}} \times \cdots \times  \Z_{2^{a_{d+1}-1}} \times \Z_{2^{a_{d+2}}}\times \cdots \times \Z_{2^{a_{d+1+t}}}.
\end{equation*}
Now $G/E$ is cyclic only if $t=0$ and $(a_1-1,a_2-1, \dots ,a_{d+1}-1)=(d+1,0, \dots,0)$, which is excluded by the assumption $\exp(G) \le 2^{d+1}$. Therefore by Theorem \ref{generalDM} there exists a $(G/E,4,1)$-difference matrix and the result follows from Theorem \ref{Main}.
\end{proof}

\begin{corollary}\label{improved}
Let $G$ be an abelian group of order $2^{2d+2}$, rank at least $d+1$, and exponent $2^e$, where 
$2 \le e \le \frac{d+3}{2}$.  
Then $G$ contains a reduced linking system of difference sets of size $2^{ \left\lfloor \frac{d+1}{e-1} \right\rfloor }-1$. 
\end{corollary}

\begin{proof}
Write $G = \Z_{2^{a_1}} \times \cdots \times \Z_{2^{a_{d+1+t}}}$, where $t \ge 0$ and $e=a_1 \ge a_2 \ge \cdots \ge a_{d+1+t} \ge 1$ and $\sum_i a_i = 2d+2$. The first $d+1$ direct factors of $G$ contain a subgroup $E$ isomorphic to $\Z_2^{d+1}$, and
\begin{equation*}
G/E \cong \Z_{2^{a_1-1}} \times \cdots \times  \Z_{2^{a_{d+1}-1}} \times \Z_{2^{a_{d+2}}}\times \cdots \times \Z_{2^{a_{d+1+t}}}.
\end{equation*}
Now $\exp(G/E) = \max(2^{a_1-1}, 2^{a_{d+2}})$, and $a_{d+2} \ge a_1$ only if $t=d+1$ and
$a_1=a_2 = \cdots = a_{2d+2} = 1$, which is excluded by the assumption $e \ge 2$. Therefore $\exp(G/E) = 2^{a_1-1}=2^{e-1}$, and so by Theorem \ref{Buratti} there exists a $(G/E, 2^{\left\lfloor \frac{d+1 }{e-1}\right\rfloor}, 1)$-difference matrix. The assumption $e \le \frac{d+3}{2}$ implies that $2^{ \left\lfloor \frac{d+1}{e-1} \right\rfloor } \ge 4$, and the result follows from Theorem \ref{Main}. 

\end{proof}

Table~\ref{Summary} in Section~\ref{sec:introduction} summarizes the constructive results of Corollaries \ref{general} and \ref{improved}, together with that of Corollary~\ref{bent=link}.
Table~\ref{Compare} in Section~\ref{sec:introduction}, and Table~\ref{256} below, illustrate the power of Corollaries \ref{general} and \ref{improved} in comparison with the best previously known results, by considering abelian groups of order 64 and 256, respectively. 
Table \ref{256} does not list those groups of order 256 having exponent 32 or rank at most 3, for which no existence results are currently known.

\begin{table}[h!]
\begin{center}
\caption{Comparison of maximum known sizes of reduced linking systems of difference sets in abelian groups of order $256$.}\label{256}
\vspace{2mm}
\begin{tabular}{c | c | c | c}
	& Previous 	& Current 	&        \\
Group   & maximum 	& maximum	& Source \\
        & known size 	& known size   	&        \\
\hline\\[-3mm]
$\Z_2^8$ & 127 \cite{bey-kyureghyan} & 127 & Bent set (Corollary \ref{bent=link}) \\
$\Z_4 \times \Z_2^6$ & None &  15 & Difference matrix (Corollary \ref{improved}) \\
$\Z_4^2 \times \Z_2^4$ & 3  \cite{davis-martin-polhill} &  15 & Difference matrix (Corollary \ref{improved}) \\
$\Z_4^3 \times \Z_2^2$ & None & 15 & Difference matrix (Corollary \ref{improved}) \\
$\Z_4^4$ & 15 \cite{davis-martin-polhill} & 15 & Difference matrix (Corollary \ref{improved}) \\
$\Z_8 \times \Z_2^5$ & None & {3} & Difference matrix  (Corollary \ref{general})\\
$\Z_8 \times \Z_4 \times \Z_2^3$ & None & {3} & Difference matrix (Corollary \ref{general}) \\
$\Z_8 \times \Z_4^2 \times \Z_2$ & None & {3} & Difference matrix (Corollary \ref{general}) \\
$\Z_8^2 \times \Z_2^2$ & None & 3 &  Difference matrix (Corollary \ref{general})\\
$\Z_{16} \times \Z_2^4$ & None &3 &  Difference matrix (Corollary \ref{general})\\
$\Z_{16} \times \Z_4 \times \Z_2^2$ & None &3 &  Difference matrix (Corollary \ref{general})
\end{tabular}
\end{center}
\end{table}

\subsection{Infinite families in nonabelian groups}\label{subsec-nonab}

We next use Theorem \ref{Main} to construct an infinite family of reduced linking systems of difference sets in nonabelian $2$-groups. No example of a linking system of difference sets in a nonabelian group was previously known. 
\begin{corollary}\label{TYKEN}
Let $d$ be a positive integer, and 
let $D_4$ be the dihedral group of order 8. Let $K$ be an abelian group of order $2^{2d-1}$ and exponent at most 4. Then $G = D_4 \times K$ contains a reduced linking system of difference sets of size $2^{d+1}-1$. 
\end{corollary}

\begin{proof}
The group $K$ has rank at least $d$, and contains a subgroup $E' \cong \Z_2^d$ such that $K / E' \cong \Z_2^{d-1}$. 
Write 
$D_4 = \langle a,b : a^4=b^2=1, \, a^{-1}=bab^{-1} \rangle$.
The center of $D_4$ is $\langle a^2 \rangle \cong \Z_2$, and $D_4/ \langle a^2 \rangle \cong \Z_2^2$. 
Therefore $E= \langle a^2 \rangle \times E'$ is a central subgroup of $G$ isomorphic to $\Z_2^{d+1}$ and 
$G/E \cong \Z_2^{d+1}$.  
By Theorem \ref{Buratti}, there exists a $(G/E,2^{d+1},1)$-difference matrix. The result follows from  Theorem~\ref{Main}.
\end{proof}

We cannot produce a reduced linking system of difference sets of larger size than in Corollary \ref{TYKEN} using the difference matrix construction of Theorem \ref{Main}, because the $(G/E, 2^{d+1},1)$-difference matrix used in its proof satisfies the extremal condition $2^{d+1}= \vert G/E \vert$ (see Section \ref{DMs}).

The technique used in the proof of Corollary~\ref{TYKEN}, of substituting the dihedral group $D_4$ for an abelian group of order~8, appears in other contexts (for example, \cite[Chapter~VI, Remarks~9.10~(b)]{bjl2}).
However, we can produce many further examples of reduced linking systems of difference sets in nonabelian $2$-groups by modifying the construction of Corollary~\ref{TYKEN}. A straightforward variation is to replace $D_4$ by the quaternion group of order~8. More sophisticated examples are readily available: by Theorems~\ref{generalDM} and~\ref{Main}, it is sufficient to find a nonabelian group $G$ of order $2^{2d+2}$ containing a central subgroup $E$ isomorphic to $\Z_2^{d+1}$ for which the factor group $G/E$ is abelian and noncyclic. The software package GAP \cite{GAP4} shows that there are 4 such groups $G$ of order 16, and 49 such groups $G$ of order~64.

\subsection{Infinite nonreversible family}\label{subsec-nonrev}

Recall from Section \ref{DMPresults} that all examples of reduced linking systems of difference sets given in \cite{davis-martin-polhill} are reversible, with the single exception of Example~\ref{FromDMP}. 
We shall generalize Example~\ref{FromDMP} to an infinite family of nonreversible examples.

We first show that Example \ref{FromDMP} can be realized using the construction \eqref{DfromDM} given in the proof of Theorem \ref{Main}. Take $G= \Z_4^2= \langle x,y \rangle$ and $E= \langle x^2,y^2 \rangle$ and $d=1$ and $m=4$, and let $H_1= \langle x^2 \rangle, H_2= \langle y^2 \rangle, H_3= \langle x^2y^2 \rangle$. Then we may represent the difference sets $D_1,D_2,D_3$ of Example~\ref{FromDMP} as
\begin{align*}
D_1 &= xH_1+yH_2+xy^3H_3 \\
D_2 &= xy H_1+xH_2+y^3H_3 \\
D_3 &= y^3H_1+x^3yH_2+xH_3,
\end{align*}
which have the form \eqref{DfromDM} where
\begin{align*} 
(b_{i,j}) = 
\begin{pmatrix}
\bone_{G}	& \bone_{G} 	& \bone_{G}	& \bone_{G}\\
\bone_{G}	& x		& y 		& xy \\
\bone_{G}	& xy 		& x 		& y\\
\bone_{G}	& y 		& xy 		& x
\end{pmatrix} \quad \text{ for } 0 \le i,j \le 3 
\end{align*}
and 
\begin{align*} 
(e_{i,j}) = 
\begin{pmatrix}
\bone_E	& \bone_E 	& y^2\\
\bone_E	& \bone_E 	& y^2 \\
y^2 	& x^2 		& 1 
\end{pmatrix}
 \quad \text{ for } 1 \le i,j \le 3.
\end{align*}
\noindent
This example has the property that $D_1$ is reversible, but neither $D_2$ nor $D_3$ is. 

We now generalize Example \ref{FromDMP}.
\begin{corollary}\label{cor-nonrev}
The group $\Z_4^{d+1}$ contains a reduced linking system of difference sets of size $2^{d+1}-1$, at least one of whose difference sets is not reversible. 
\end{corollary}

\begin{proof}
Let $G = \Z_4^{d+1} = \langle x_1, x_2 , \dots, x_{d+1} \rangle$ and $E = \langle x_1^2, x_2^2 , \dots, x_{d+1}^2 \rangle \cong \Z_2^{d+1}$. Then $G/E \cong \Z_2^{d+1}$ and by Theorem \ref{Buratti} there is a 
$(G/E, 2^{d+1},1)$-difference matrix $(b_{i,j}E)$, where each $b_{i,j}E$ has the form 
$\prod_{r \in R(i,j)} x_r E$ for some subset $R(i,j)$ of $\{1,2, \dots, d+1\}$. Following the proof of Theorem \ref{Main}, let $s=2^{d+1}-1$ and let $H_1,H_2, \dots, H_s$ be the subgroups of $G$ corresponding to hyperplanes of $E$ when $E$ is regarded as a vector space of dimension $d+1$ over $\GF(2)$. We may take $H_1 = \langle x_2^2, x_3^2, \dots, x_{d+1}^2 \rangle$. As discussed after Definition~\ref{DMdef}, we may also assume that for each $i \ge 1$, the set $\{b_{i,j}E : 0 \le j \le s \}$ contains no repeated element and that $b_{1,1}=x_1$. Now define $D_i$ as in \eqref{DfromDM} for $m=2^{d+1}$, taking $e_{1,1}= \bone_E$. This gives a reduced linking system of difference sets in $G$ of size $2^{d+1}-1$.

We now show that 
\begin{align*}
D_1 = x_1H_1+ \sum_{j=2}^s b_{1,j}e_{1,j}H_j
\end{align*}
is not reversible. Since $x_1 \in D_1$, it is sufficient to show that $x_1^3 \not \in D_1$. Suppose, for a contradiction, that $x_1^3 \in D_1$. Since $x_1^2 \not \in H_1$, this implies that 
$x_1^3 \in b_{1,j}e_{1,j} H_j$ for some $j>1$. 
But $e_{1,j} H_j \subset E = \langle x_1^2, x_2^2 , \dots, x_{d+1}^2 \rangle$, so by considering the parity of the exponent of each $x_r$ in $b_{1,j}E = \prod_{r \in R(1,j)} x_r E$, we conclude that $b_{1,j} \in x_1E$ for some $j>1$. This contradicts that the set $\{b_{1,j}E : 0 \le j \le s \}$ contains no repeated element. 
\end{proof}

\subsection{The group $\Z_4^2$}\label{subsec-Z42}
We further illustrate the strength of the difference matrix construction of Theorem \ref{Main} by examining  reduced linking systems of difference sets in $\Z_4^2= \langle x,y \rangle$ of size 3. 
Let $m=4$ and $E= \langle x^2, y^2 \rangle$ and $H_1 = \langle x^2 \rangle, H_2 = \langle y^2 \rangle, H_3 = \langle x^2 y^2 \rangle$, so that \eqref{DfromDM} becomes
\begin{align}
D_1 &= b_{1,1} e_{1,1} \langle x^2 \rangle + b_{1,2} e_{1,2} \langle y^2 \rangle + b_{1,3} e_{1,3} \langle x^2 y^2 \rangle \nonumber \\ 
D_2 &= b_{2,1} e_{2,1} \langle x^2 \rangle + b_{2,2} e_{2,2} \langle y^2 \rangle + b_{2,3} e_{2,3} \langle x^2 y^2 \rangle \label{ForCounting} \\
D_3 &= b_{3,1} e_{3,1} \langle x^2 \rangle + b_{3,2} e_{3,2} \langle y^2 \rangle + b_{3,3} e_{3,3} \langle x^2 y^2 \rangle. \nonumber
\end{align}

We first show that there are at least $2^{16}$ distinct reduced linking systems $\{D_1,D_2,D_3 \}$ of this form; an exhaustive computer search shows that this accounts for all reduced linking systems of difference sets in $\Z_4^2$ of size 3, and that no larger system exists.

Each $e_{i,j} \in E$ can be chosen arbitrarily, and exactly 2 of the 4 choices for each $e_{i,j}$ give distinct values for the coset $e_{i,j} H_j$. This counts $2^9$ choices. 
The matrices 
\begin{align*} 
(b_{i,j}) = 
\begin{pmatrix}
\bone_{G}	& \bone_{G} 	& \bone_{G}	&\bone_{G} \\
\bone_{G}	& x		& y 		& xy \\
\bone_{G} 	& y 		& xy 		& x \\
\bone_{G} 	& xy	 	& x 		& y
\end{pmatrix}
\text{ and }
(b'_{i,j}) = 
\begin{pmatrix}
\bone_{G} 	& \bone_{G} 	& \bone_{G}	&\bone_{G} \\
\bone_{G}	& x		& xy 		& y \\
\bone_{G} 	& y 		& x	 	& xy \\
\bone_{G} 	& xy 		& y 		& x
\end{pmatrix}
\end{align*}
for $0 \le i,j \le 3$ correspond to $(G/E,4,1)$-difference matrices $(b_{i,j}E)$ and $(b'_{i,j}E)$. We can multiply all entries of a row of either $(G/E,4,1)$-difference matrix by a fixed $a \in \{\bone_{G}E,xE,yE,xyE\}$ without changing the defining property of the difference matrix. This gives $4^3$ possible row multiples for rows $1,2,3$ of each of the matrices $(b_{i,j}E)$ and $(b'_{i,j}E)$, and so counts $2 \cdot 4^3=2^7$ choices. Moreover, we see from \eqref{ForCounting} that each of the resulting 
$2^9 \cdot 2^7=2^{16}$ choices gives a distinct reduced linking system $\{D_1,D_2,D_3 \}$.
(We cannot directly compare this count of reduced linking systems of difference sets with the classification of systems of linked symmetric $(16,6,2,4)$ designs given by Mathon \cite{mathon}, which counts the number of isomorphism classes rather than the number of distinct systems.)

We next consider the reversibility of these $2^{16}$ reduced linking systems of difference sets. We have already seen an example in Section \ref{subsec-nonrev} (namely Example~\ref{FromDMP}) for which exactly one of the three difference sets is reversible. We can readily specify a reduced linking system for which none of the difference sets is reversible, for example by taking 
\begin{align*} 
(b_{i,j}) = 
\begin{pmatrix}
\bone_{G} 	& \bone_{G} 	& \bone_{G}	&\bone_{G} \\
\bone_{G} 	& x		& xy 		& y \\
\bone_{G} 	& y 		& x		& xy \\
\bone_{G} 	& xy	 	& y 		& x
\end{pmatrix}
\text{ and }
(e_{i,j}) \text{ arbitrary}, 
\end{align*}
and a reduced linking system for which all three difference sets are reversible, for example by taking 
\begin{align*} 
(b_{i,j}) = 
\begin{pmatrix}
\bone_{G} 	& \bone_{G} 	& \bone_{G}	&\bone_{G} \\
x 		& \bone_{G}	& y 		& xy \\
y 		& x 		& \bone_{G}	& xy \\
xy	 	& x	 	& y 		& \bone_{G}
\end{pmatrix}
\text{ and }
(e_{i,j}) \text{ arbitrary}. 
\end{align*}

\section{Open problems}
We conclude with some open problems.
\label{sec:open}
\begin{enumerate}[1.]
\item Our main constructive result (Theorem~\ref{Main}) uses difference matrices in $2$-groups as a crucial ingredient. Are there examples of difference matrices with more rows than those specified in Theorems~\ref{generalDM} and~\ref{Buratti}, or in other $2$-groups?
If so, this would immediately give new reduced linking systems of difference sets.

\item Our main constructive result depends on Lemma~\ref{newconstruction}, involving hyperplanes. Following Q4 of Section \ref{overview}, is there a construction for reduced linking systems of difference sets involving another combinatorial object such as a partial difference set?

\item Table~\ref{Summary} extends some of the previous results due to Davis, Martin, and Polhill \cite{davis-martin-polhill} for abelian $2$-groups (Theorem \ref{thm-dmp}), but does not explain all their results. Can the constructive framework of this paper be broadened to do so?

\item There is a recursive construction for difference sets in the five known families whose parameters satisfy $\gcd(v,n)>1$ (namely the Hadamard, McFarland, Spence, Davis-Jedwab, and Chen families) \cite{unify}, \cite{chen-newfam}. Is there an analogous recursive construction for reduced linking systems of difference sets?

\item Q2 of Section \ref{overview} asks whether there is a linking system of difference sets in a non-$2$-group, and this question remains open despite the nonexistence results of Section~\ref{sec:nonexistence} and the constructions recently given by Kodalen~\cite{kodalen}.
Can this question be resolved constructively, or else can its scope be narrowed by finding further nonexistence results similar to Theorems~\ref{NoMcFarland} and~\ref{NoSpence}?

\item We have counted all reduced linking systems of difference sets in the group $\Z_4^2$ having maximum size (Section \ref{subsec-Z42}). Can this counting result be extended to other groups?

\end{enumerate}

\section*{Acknowledgements}
We greatly appreciate Ken Smith's kind assistance in formulating Corollary~\ref{TYKEN}.
We are very grateful to Petr Lisonek and Marni Mishna for their careful and helpful comments on the manuscript.


\end{document}